\documentclass[11pt]{amsart}

\usepackage{latexsym,rawfonts}
\usepackage{amsfonts,amssymb}
\usepackage{amsmath,amsthm}

\usepackage[plainpages=false]{hyperref}
\usepackage{graphicx}
\usepackage{color}
\usepackage[table]{xcolor}
\usepackage{longtable}

\newcommand{\R}{\mathbb{R}}

\newcommand{\ko}{\mathcal K_o^n}
\newcommand{\koo}{\mathcal K^n}

\newcommand{\rnnn}{\mathbb R^{n}}

\newcommand{\sn}{ {\mathbb{S}^{n-1}}}

\newcommand{\psum}{{+_{\negthinspace\kern-2pt p}}\,}
\newcommand{\qsum}[1]{{+_{\negthinspace\kern-2pt #1}}\,}
\newcommand{\dpsum}{{\tilde+_{\negthinspace\kern-1pt p}}\,}
\newcommand{\dqsum}[1]{{\tilde+_{\negthinspace\kern-1pt #1}}\,}
\newcommand{\lsub}[1]{\hskip -1.5pt\lower.5ex\hbox{$_{#1}$}}

\numberwithin{equation}{section}

\newtheorem{theo}{Theorem}[section]

\newtheorem{lem}[theo]{Lemma}

\newtheorem{rem}[theo]{Remark}

\newtheorem{lemma}[theo]{Lemma}

\theoremstyle{definition}
\newtheorem{defi}[theo]{Definition}

\begin{document}

\title{The generalized Gaussian log-Minkowski problem}

\author[J. Hu]{Jinrong Hu}
\address{School of Mathematics, Hunan University, Changsha, 410082, Hunan Province, China}

\email{hujinrong@hnu.edu.cn}

\begin{abstract}
The generalized Gaussian distribution that stems from information theory is studied.  The log-Minkowski problem  associated with generalized Gaussian distribution shall be introduced and solved.
\end{abstract}
\keywords{Generalized Gaussian volume, log-Minkowski problem}

\subjclass[2010]{52A20, 52A40}

\thanks{The research is supported, in part, by the National Science Foundation of China (12171144, 12231006)}

\maketitle

\baselineskip18pt

\parskip3pt

\section{Introduction}
\label{Sec1}

Given a convex body $K$ in $\rnnn$, the Gaussian volume $\gamma_{n}(K)$ in the Gaussian probability space is defined by
\[
\gamma_{n}(K)=\frac{1}{(\sqrt{2\pi})^{n}}\int_{K}e^{-\frac{|x|^{2}}{2}}dx.
\]
Unlike Lebesgue measure, Gaussian volume is neither translation invariant nor homogeneous. Recently, Huang-Xi-Zhao \cite{HXYZ21} established the variational formula of $\gamma_{n}(\cdot)$ as
\[
\lim_{t\rightarrow 0}\frac{\gamma_{n}(K+tL)-\gamma_{n}(K)}{t}=\int_{\sn}h_{L}dS_{\gamma_{n}}(K,\cdot),
\]
where the Gaussian surface measure $S_{\gamma_{n}}(K,\cdot)$ is defined by
\begin{equation*}\label{tormes22}
S_{\gamma_{n}}(K,\eta)=\frac{1}{(\sqrt{2\pi})^{n}}\int_{\nu^{-1}_{K}(\eta)}e^{-\frac{|x|^{2}}{2}}{d}\mathcal{H}^{n-1}(x)
\end{equation*}
for every Borel subset $\eta\subset {\sn}$. Here $\mathcal{H}^{n-1}$ is the $(n-1)$-dimensional Hausdorff measure, $\nu_{K}$ is the Gauss map defined on the subset of those points of $\partial K$, and $\nu^{-1}_{K}$ is the inverse Gauss map (see Sec. \ref{Sec2} for a precise definition). Huang-Xi-Zhao \cite{HXYZ21} proposed the Gaussian Minkowski problem prescribing Gaussian surface measure.  At the same time, the authors  \cite{HXYZ21} obtained the normalized solution to the Gaussian Minkowski problem by applying the variational arguments \cite{HLYZ10,HLYZ16} and moreover derived the existence of weak solution to the non-normalized Gaussian Minkowski problem by using degree theory method. Subsequently, the development on the Gaussian Minkowski problem flourishes, see e.g., \cite{CHW23,FH23,IE23,KL23,Lu22,Sh23}.

As a generalization of standard Gaussian distribution, the generalized Gaussian distribution holds a central role concerning classical information measures. In information theory,  this distribution  appears naturally as an extremal distribution for various moment, entropy,  Fisher information and Cram\'{e}r-Rao inequality. As introduced by Lutwak-Lv-Yang-Zhang \cite{ELYZ12}, the generalized Gaussian volume is defined as follows:
\begin{defi}\label{definition}

For $m>0$ and $b<\frac{m}{n}$, the generalized Gaussian  volume of a convex body $K$ in $\rnnn$, $\gamma_{b,m}(K)$, is denoted by
 \[
 \gamma_{b,m}(K)=\int_{K}g_{b,m}(x)dx,
  \]
where the generalized Gaussian distribution (density function) $g_{b,m}$ is given by
\begin{equation*}
		g_{b,m}(x)=
		\begin{cases}
		q_{b,m}[1-\frac{b}{m}|x|^{m}]_+^{\frac{1}{b}-\frac{n}{m}-1},\qquad &{\rm if}\ b\neq 0,\\
			q_{0,m}e^{-\frac{1}{m}|x|^m},\qquad &{\rm if}\ b=0,
		\end{cases}
	\end{equation*}
where $t_{+}=\max(t,0)$ and
\begin{equation*}
		q_{b,m}=
		\begin{cases}
			\frac{\frac{m}{n}|\frac{b}{m}|^{\frac{n}{m}}\Gamma(\frac{n}{2}+1)}{\pi^{\frac{n}{2}}B(\frac{n}{m},1-\frac{1}{b})},\qquad &{\rm if}\  b<0,\\
			\frac{\Gamma(\frac{n}{2}+1)}{\pi^{\frac{n}{2}}m^{\frac{n}{m}}\Gamma(\frac{n}{m}+1)},\qquad &{\rm if}\ b=0,\\
			\frac{\frac{m}{n}(\frac{b}{m})^{\frac{n}{m}}\Gamma(\frac{n}{2}+1)}{\pi^{\frac{n}{2}}B(\frac{n}{m},\frac{1}{b}-\frac{n}{m})},\qquad &{\rm if}\ b>0,
		\end{cases}
	\end{equation*}
	where $\Gamma$ is the gamma function and $B$ is the Beta function.
Notice that $q_{b,m}$ is a normalization constant such that $g_{b,m}$ is a probability density. The above generalized Gaussian distribution includes the standard Gaussian distribution ($b=0,m=2$) and $t$-student distribution.
\end{defi}
Very recently, Liu-Tang \cite{JT24} established the variational formula of generalized Gaussian volume under $L_{p}$ sum with $p\neq 0$, which yields the generalized $L_{p}$ $(p\neq 0)$ Gaussian surface area measure as
\begin{equation}\label{JU}
	S_{p,b,m}(K,\eta)=\frac{1}{p}\int_{\nu^{-1}_{K}(\eta)}(x\cdot \nu_{K}(x))^{1-p}g_{b,m}(x)d\mathcal{H}^{n-1}(x)
	\end{equation}
for each Borel subset $\eta\subset \sn$. Here $S_{1,b,m}(K,\cdot)$ is abbreviated as $S_{b,m}(K,\cdot)$. Employing similar techniques in \cite{HXYZ21}, the authors \cite{JT24} obtained the solvability of the generalized $L_{p}$ Gaussian Minkowski problem for $p\geq 1$.

We are in a position to define the generalized Gaussian cone measure associated with the generalized Gaussian distribution.

\begin{defi} \label{def}
Let $K$ be a convex body in $\rnnn$. The generalized Gaussian cone measure of $K$, denoted by $G_{b,m}(K,\eta)$, is given by
\begin{equation}\label{JU}
		G_{b,m}(K,\eta)=\int_{\nu^{-1}_{K}(\eta)}x\cdot \nu_{K}(x)g_{b,m}(x)d\mathcal{H}^{n-1}(x)
	\end{equation}
for each Borel subset $\eta\subset \sn$.
\end{defi}
The Minkowski problem prescribing the generalized Gaussian cone measure is:

{\bf The generalized Gaussian log-Minkowski Problem.} Given a nonzero finite Borel measure $\mu$ on $\sn$, what are the necessary and sufficient conditions on $\mu$ so that there exists a convex body $K\subset \rnnn$ containing the origin such that
\[
G_{b,m}(K,\cdot)=\mu?
\]

The solvability of (normalized) generalized Gaussian log-Minkowski problem is obtained via a variational argument as follows.

\begin{theo}\label{mtheo}
Suppose $b< \frac{m}{n+m}$.  If $\mu$ is a nonzero finite Borel measure on $\sn$ that is not concentrated in any closed hemisphere, then there exists a convex body $K\in \ko$ with $\gamma_{b,m}(K)>1/2$
such that
\[
\frac{\mu}{|\mu|}= \frac{G_{b,m}(K,\cdot)}{G_{b,m}(K,\sn)}.
\]

\end{theo}

\begin{rem}
Notice that Theorem \ref{mtheo} is the normalized version of the solvability of the generalized Gaussian log-Minkowski problem. Due to the non-homogeneity of the generalized Gaussian cone measure, we can not directly obtain the solvability of (non-normalized) generalized Gaussian log-Minkowski problem by Theorem \ref{mtheo}. It is very interesting to explore the variational structure of solving the generalized Gaussian log-Minkowski problem. In addition, Theorem \ref{mtheo} is interpreted as the solvability of Gaussian log-Minkowski problem in the context of $b=0$ and $m=2$, which was previously studied by \cite{Sh23} by virtue of the subspace concentration condition in \cite{BLYZ12}.
\end{rem}

\section{Preliminaries}
\label{Sec2}

In this section, we provide fundamental facts about convex body. Some standard  references are recommended,  such as books of Gardner \cite{G06} and Schneider \cite{S14}.

Denote by ${\rnnn}$ the $n$-dimensional Euclidean space, by $o$ the origin of ${\rnnn}$. The unit ball centered at $o$ in $\rnnn$ is written by $B^{n}$, its boundary by ${\sn}$. We write $\omega_{n}$ for the $n$-dimensional volume of $B^{n}$.   For $x,y\in {\rnnn}$, $x\cdot y$ denotes the standard inner product.  For $x\in{\rnnn}$, denote by $|x|=\sqrt{x\cdot x}$ the Euclidean norm.  Denote by $C({\sn})$ the set of continuous functions on the unit sphere ${\sn}$, and by $C^{+}({\sn})$ the set of strictly positive functions in $C({\sn})$.  A compact convex set of ${\rnnn}$ with non-empty interior is called as a convex body. The set of all convex bodies in $\rnnn$ is denoted by $\koo$. The set of all convex bodies containing the origin in the interior is denoted by $\ko$.

 Given $x\in{\rnnn}$, the support function of a compact convex set $K$ is defined by
\[
h_{K}(x)=\max\{x\cdot y:y \in K\}.
\]
Suppose $K$ contains the origin in its interior, the radial function $\rho_{K}$ with respect to the origin is defined by
\[
\rho_{K}(x)=\max\{\lambda:\ \lambda x\in K\}, \ x\in \rnnn \backslash \{0\}.
\]
It is clear to see that $\rho_{K}(u)u\in \partial K$ for all $u\in \sn$.

For compact convex sets $K$ and $L$ in ${\rnnn}$, any real $a_{1},a_{2}\geq 0$, define the Minkowski sum of $a_{1}K+a_{2}L$ in ${\rnnn}$ by
\[
a_{1}K+a_{2}L=\{a_{1}x+a_{2}y:x\in K,\ y\in L\},
\]
and its support function is given by
\[
h_{{a_{1}K+a_{2}L}}(\cdot)=a_{1}h_{K}(\cdot)+a_{2}h_{L}(\cdot).
\]

The $L_{p}$ sum of $K,L\in \ko$ for $a,b>0$ and $p\neq 0$ is defined as (see \cite{F62})
\[
a\cdot K+_{p} b\cdot L= \bigcap_{v\in \sn}\left\{x\in \rnnn:x\cdot v\leq (ah_{K}(v)^{p}+bh_{L}(v)^{p})^{\frac{1}{p}}\right\}
\]
and the $L_{0}$ sum (log-Minkowski sum) of $K,L\in \ko$ for $a,b>0$ is defined as
\[
a\cdot K+_{0}b\cdot L=\bigcap_{v\in \sn}\left\{x\in \rnnn:x\cdot v\leq h_{K}(v)^{a}h_{L}(v)^{b} \right\}.
\]

Given a compact convex set $K$ in $\rnnn$, for $\mathcal{H}^{n-1}$ almost all $x\in \partial K$, the unit outer normal of $K$ at $x$ is unique. In such case, we use $\nu_{K}$ to denote the Gauss map  that takes $x\in \partial K$ to its unique unit outer normal.
We write $\nu^{-1}_{K}$ for the inverse Gauss map.

For those $u\in \sn$ such that $\nu_{K}$ is well-defined at $\rho_{K}(u)u\in \partial K$, we write $\alpha_{K}(u)$ for $\nu_{K}(\rho_{K}(u)u)$.

 The Hausdorff metric $\mathcal{D}(K,L)$ between two compact convex sets $K$ and $L$ in ${\rnnn}$, is expressed as
\[
\mathcal{D}(K,L)=\max\{|h_{K}(v)-h_{L}(v)|:v\in {\sn}\}.
\]
Let $K_{j}$ be a sequence of compact convex sets in ${\rnnn}$, for a compact convex set $K_{0}$ in ${\rnnn}$, if $\mathcal{D}(K_{j},K_{0})\rightarrow 0$, then $K_{j}$ converges to $K_{0}$.

For each $h\in C^{+}(\sn)$, the \emph{Wulff shape} $[h]$ generated by $h$,  is the convex body defined by
\[
[h]=\{x\in \rnnn: x \cdot v\leq h(v), \ {\rm for \ all} \ v\in \sn\}.
\]



\section{Generalized Gaussian cone measure and the associated optimization problem}
\label{Sec4}

We first get the variational formula of generalized Gaussian volume under Minkowski sum to yield the generalized Gaussian surface area measure. The following lemma is needed.
\begin{lem}\label{KON}
Let $K\in \ko$. Suppose that $f:\sn\rightarrow \R$ is a continuous function
and $\delta>0$. Let $h_{t}:\sn \rightarrow (0,\infty)$ be a continuous function defined for each $t\in (-\delta,\delta)$ by
\[
h_{t}=h_{K}+tf+o(t,\cdot), \ on \  \sn,
\]
where $o:(-\delta,\delta)\times \sn \rightarrow \R$ is such that $o(t,\cdot): \sn \rightarrow \R$ is continuous, for each $t$, and  $o(t,\cdot)/t\rightarrow 0$ uniformly on $\sn$ as $t\rightarrow 0$. Then,
\[
\lim_{t\rightarrow 0}\frac{\rho_{[h_{t}]}(u)-\rho_{K}(u)}{t}=\frac{f(\alpha_{K}(u))}{h_{K}(\alpha_{K}(u))}\rho_{K}(u)
\]
for almost all $u\in \sn$ with respect to spherical Lebesgue measure. Moreover, there exists $M>0$ such that
\[
|\rho_{[h_{t}]}(u)-\rho_{K}(u)|\leq M|t|,
\]
for all $u\in \sn$ and $t\in (-\delta,\delta)$.
\end{lem}
\begin{proof}
The desired results follow from Lemma 2.8, Lemma 4.1, Lemma 4.3 in \cite{HLYZ16}, also the facts that
\[
\log h_{t}=\log h_{K}+t \frac{f}{h_{K}}+o(t,\cdot)
\]
and that
\[
|s-1|\leq M_{0}|\log s|, \ {\rm when} \ s\in (0,M_{0})
\]
for a positive constant $M_{0}$.
\end{proof}
The variational formula of the generalized Gaussian volume under Minkowski sum produces the generalized Gaussian surface area measure, which is shown as follows.
\begin{theo}\label{yt}
Suppose $b< \frac{m}{n+m}$. Let $K\in \ko$ and $f\in C(\sn)$. Then,
\begin{equation*}\label{SVA}
\lim_{t\rightarrow 0}\frac{\gamma_{b,m}([h_{t}])-\gamma_{b,m}(K)}{t}=\int_{\sn}fdS_{b,m}(K).
\end{equation*}
\end{theo}
\begin{proof}
Case I: If $b=0$. Using the polar coordinate, we obtain
\[
\gamma_{0,m}([h_{t}])=q_{0,m}\int_{\sn}\int^{\rho_{[h_{t}]}(u)}_{0}e^{-\frac{1}{m}r^{m}}r^{n-1}drdu.
\]
Since $K\in \ko$ and $f\in C(\sn)$, as $t$ closes to 0, there exists $M_{1}>0$ such that $[h_{t}]\subset M_{1}B^{n}$. Set $F(s)=\int^{s}_{0}e^{-\frac{1}{m}r^{m}}r^{n-1}dr$. By virtue of mean value theorem and Lemma \ref{KON}, we get
\[
|F(\rho_{[h_{t}]}(u))-F(\rho_{K}(u))|\leq |F^{'}(\theta)||\rho_{[h_{t}]}(u)-\rho_{K}(u)|\leq M |F^{'}(\theta)||t|,
\]
where $M$ is given by  Lemma \ref{KON}, and $\theta$ is between $\rho_{[h_{t}]}(u)$ and $\rho_{K}(u)$. Due to $[h_{t}]\subset M_{1}B^{n}$, we know that $\theta\in (0,M_{1})$. Thus, from the definition of $F$, we conclude that $|F^{'}(\theta)|$ is bounded from above by some constants that depend on $M_{1}$. Therefore, there exists $M_{2}>0$ such that
\[
|F(\rho_{[h_{t}]}(u))-F(\rho_{K}(u))|\leq M_{2}|t|.
\]
In conjunction with dominated convergence theorem, and by again Lemma \ref{KON}, we get
\begin{equation*}
\begin{split}
\label{Up3}
\lim_{t\rightarrow 0}\frac{\gamma_{0,m}([h_{t}])-\gamma_{0,m}(K)}{t}&=q_{0,m}\int_{\sn}f(\alpha_{K}(u))e^{-\frac{1}{m}\rho^{m}_{K}(u)}\frac{\rho_{K}(u)^{n}}{h_{K}(\alpha_{K}(u))}du\\
&=q_{0,m}\int_{\partial K}f(\nu_{K}(x))e^{-\frac{1}{m}|x|^{m}}d\mathcal{H}^{n-1}(x)\\
&=\int_{\sn}fdS_{0,m}(K).
\end{split}
\end{equation*}

Case II: If $b<0$ or $0<b< \frac{m}{n+m}$. Using again the polar coordinate, we also have
\begin{equation*}
\gamma_{b,m}([h_t])=q_{b,m}\int_{\sn}\int_{0}^{\rho_{[h_t]}(u)}(1-\frac{b}{m} r^m)_+^{\frac{1}{b}-\frac{n}{m}-1}r^{n-1}drdu.
\end{equation*}
Similarly, due to $K\in \ko$ and $f\in C(\sn)$, as $t$ closes to 0, there exists $M_{1}>0$ such that $[h_{t}]\subset M_{1}B^{n}$. Set $\widetilde{F}(s)=\int^{s}_{0}(1-\frac{b}{m}r^{m})^{\frac{1}{b}-\frac{n}{m}-1}_{+}r^{n-1}dr$. Applying mean value theorem and Lemma \ref{KON}, we get
\[
|\widetilde{F}(\rho_{[h_{t}]}(u))-\widetilde{F}(\rho_{K}(u))|\leq |\widetilde{F}^{'}(\theta)||\rho_{[h_{t}]}(u)-\rho_{K}(u)|\leq M |\widetilde{F}^{'}(\theta)||t|,
\]
where $M$ is given by Lemma \ref{KON}, and $\theta$ is between $\rho_{[h_{t}]}(u)$ and $\rho_{K}(u)$. Due to $[h_{t}]\subset M_{1}B^{n}$, we know that $\theta\in (0,M_{1})$. Thus, from the definition of $\widetilde{F}$, we conclude that $|\widetilde{F}^{'}(\theta)|$ is bounded from above by some constants that depend on $M_{1}$ for $b<0$ or $0<b< \frac{m}{n+m}$. Therefore, there exists $M_{2}>0$ such that
\[
|\widetilde{F}(\rho_{[h_{t}]}(u))-\widetilde{F}(\rho_{K}(u))|\leq M_{2}|t|.
\]
By the dominated convergence theorem, and by again Lemma \ref{KON}, we get
\begin{equation*}
\begin{split}
\label{Up3}
\lim_{t\rightarrow 0}\frac{\gamma_{b,m}([h_{t}])-\gamma_{b,m}(K)}{t}&=q_{b,m}\int_{\sn}f(\alpha_{K}(u))(1-\frac{b}{m}\rho^{m}_{K}(u))^{\frac{1}{b}-\frac{n}{m}-1}_{+}\frac{\rho_{K}(u)^{n}}{h_{K}(\alpha_{K}(u))}du\\
&=q_{b,m}\int_{\partial K}f(\nu_{K}(x))(1-\frac{b}{m}|x|^{m})^{\frac{1}{b}-\frac{n}{m}-1}_{+}d\mathcal{H}^{n-1}(x)\\
&=\int_{\sn}fdS_{b,m}(K).
\end{split}
\end{equation*}
\end{proof}

Based on Theorem \ref{yt}, we can get the variational formula of generalized Gaussian volume under log-Minkowski sum to produce the generalized Gaussian cone measure.

\subsection{Generalized Gaussian cone measure}
\begin{theo}\label{yt2}
Suppose $b< \frac{m}{n+m}$. Let $K\in \ko$. Suppose that $f:\sn \rightarrow \R$ is a continuous function and $\delta>0$. For each $t\in (-\delta,\delta)$, define
\[
\log h_{t}=\log h_{K}+tf+o(t,\cdot), \quad on \ \sn,
\]
where $o:(-\delta,\delta)\times \sn \rightarrow \R$ is such that $o(t,\cdot):\sn \rightarrow \R$ is continuous, for each $t$, and $o(t,\cdot)/t\rightarrow 0$ uniformly on $\sn$ as $t\rightarrow 0$.
Then,
\begin{equation}\label{iue}
\lim_{t\rightarrow 0}\frac{\gamma_{b,m}([h_{t}])-\gamma_{b,m}(K)}{t}=\int_{\sn}fdG_{b,m}(K).
\end{equation}
\end{theo}
\begin{proof}
It is clear to see
\[
h_{t}=h_{K}+tfh_{K}+o(t,\cdot),\quad {\rm on} \ \sn.
\]
Thus, \eqref{iue} follows immediately from Theorem \ref{yt}.
\end{proof}

Based on above variational formulas, we can construct the subsequent optimization problem.

\subsection{The optimization problem}
\
\newline
\indent
Given a nonzero finite Borel measure $\mu$ on $\sn$, we first define the function $\Phi_{\mu}:C^{+}(\sn)\rightarrow \R$ by
\begin{equation*}\label{}
\Phi_{\mu}(h)=\int_{\sn}\log h(v)d\mu(v).
\end{equation*}
 If the above function $\Phi_{\mu}$ is restricted to the support function of a convex body in $\ko$, then $\Phi_{\mu}$ can be expressed as a function on $\ko$,
 \[
 \Phi_{\mu}:\ko\rightarrow \R,
 \]
 which is given by
\begin{equation}\label{phi}
 \Phi_{\mu}(K):=\int_{\sn}\log h_{K}(v)d\mu(v).
\end{equation}
Clearly, $\Phi_{\mu}(K)=\Phi_{\mu}(h_{K})$.

 We are in a position to reveal the equivalence of two extreme problems associated with $\Phi_{\mu}(h)$ and $\Phi_{\mu}(K)$ as follows.

\begin{lem}\label{EMK}
Suppose $\mu$ is a nonzero finite Borel measure on $\sn$.  For $\kappa_0\in (3/4,1)$. If there exists $K_{0}\in \ko$ with $\gamma_{b,m}(K_{0})=\kappa_{0}$ such that
\begin{equation}\label{eq 1}
\Phi_{\mu}(K_{0})=\inf\{\Phi_{\mu}(K):\ \gamma_{b,m}(K)=\kappa_{0}\ {\rm and} \ K\in \ko\},
\end{equation}
then
\begin{equation}\label{eq 2}
\Phi_{\mu}(h_{K_{0}})=\inf\{\Phi_{\mu}(h):  \gamma_{b,m}([h])=\kappa_{0} \ {\rm and} \ h\in C^{+}(\sn)\}.
\end{equation}
\end{lem}
\begin{proof}
Let $h\in C^{+}(\sn)$, define the \emph{Wulff shape} as
\[
[h]=\bigcap_{v\in \sn}\{x\in \rnnn: x\cdot v\leq h(v)\},
\]
clearly, $[h]\in \ko$. Since $h_{[h]}\leq h$ and  $[h_{[h]}]=[h]$, by the definition of $\Phi_{\mu}$, we have
\[
\Phi_{\mu}(h_{[h]})=\Phi_{\mu}([h])\leq \Phi_{\mu}(h).
\]
Hence,  $\Phi_{\mu}(h_{K_{0}})\leq \Phi_{\mu}(h)$ for all $h\in C^{+}(\sn)$ with $\gamma_{b,m}([h])=\kappa_{0}$.  The proof is completed.
\end{proof}

The next lemma implies that a solution to above extremal problem \eqref{eq 1} is a solution to the non-symmetric (normalized) generalized Gaussian log-Minkowski problem.

\begin{lemma}\label{lemma euler lagrange}
	If the minimum of \eqref{eq 1} is attained at $K_0\in \ko$, then
	\begin{equation*}
		\frac{\mu}{|\mu|} = \frac{G_{b,m}(K_0, \cdot)}{G_{b,m}(K_0,\sn)}.
	\end{equation*}
\end{lemma}
\begin{proof}
 With the aid of Lemma \ref{EMK}, we know that
\begin{equation}\label{OP2}
\Phi_{\mu}(K_{0})=\inf\{\Phi_{\mu}(h):  \gamma_{b,m}([h])=\kappa_{0}\in (3/4,1) \ {\rm and} \ h\in C^{+}(\sn)\}.
\end{equation}

	Let $f\in C(\sn)$ be arbitrary and for sufficiently small $|t|<\delta$ with $\delta>0$, set
	\begin{equation*}
		h_t= h_{K_{0}}e^{tf}.
	\end{equation*}
	Consider the function $P: [0,\infty)\times (-\delta,\delta) \rightarrow [0,1]$ given by
	\begin{equation*}
		P(\lambda, t) = \gamma_{b,m}([\lambda h_t]).
	\end{equation*}
	Note that $P(1, 0)= \kappa_{0}$ and by using Theorem \ref{yt}, we derive
	\begin{equation*}
		\partial_\lambda P(1,0) = \int_{\sn} h_{K_0}(v)dS_{b,m}(K_0,v)=G_{b,m}(K_{0},\sn)>0.
	\end{equation*}
	Thus, by the \emph{Implicit Function Theorem},  for sufficiently small $|t|$, there exists $\lambda(t)\in [1/2, 3/2]$ such that  $\gamma_{b,m}([\lambda(t)h_t])=\kappa_{0}$, $\lambda(0)=1$, and $\lambda(t)$ is differentiable at $0$, moreover, by Theorem \ref{yt2}, there is
	\begin{equation}\label{mt2}
		\lambda'(0) = -\frac{\int_{\sn}f(v) dG_{b,m}(K_0,v)}{\int_{\sn}h_{K_0}(v)dS_{b,m}(K_0,v)}= -\frac{\int_{\sn}f(v) dG_{b,m}(K_0,v)}{G_{b,m}(K_0,\sn)} .
	\end{equation}
	Now, set $g_t = \lambda(t)h_t$. Then, $g_t\in C^{+}(\sn)$ and $\gamma_{b,m}([g_t])=\kappa_{0}$. Since $h_{K_0}$ is a minimizer to \eqref{OP2}, and by using \eqref{mt2}, we have
	\begin{equation*}
		\begin{aligned}
			0 &= \left.\frac{d}{dt}\right|_{t=0} \int_{\sn} \log g_td\mu \\
			&= \left.\frac{d}{dt}\right|_{t=0} \log \lambda(t)|\mu|+\left.\frac{d}{dt}\right|_{t=0} \int_{\sn} (\log h_{K_0}+tf)d\mu\\
			&= \lambda'(0)|\mu | + \int_{\sn} fd\mu \\
			&= -\frac{\int_{\sn}f(v) dG_{b,m}(K_0,v)}{G_{b,m}(K_0,\sn)}|\mu|+ \int_{\sn} fd\mu.
		\end{aligned}
	\end{equation*}	
	Since $f\in C(\sn)$ is arbitrary, we have
	\begin{equation*}
		\frac{\mu}{|\mu|} = \frac{G_{b,m}(K_0, \cdot)}{G_{b,m}(K_0,\sn)}.
	\end{equation*}
The proof is completed.
\end{proof}

\section{The entropy estimate}
\label{Sec6}

This section is dedicated to providing entropy estimation for $\Phi_{\mu}(K)$, which plays an important role in solving the generalized Gaussian log-Minkowski problem. We first do some preparations.

Given a convex body $K \in \ko$, we shall  use the notions $R_{K}$ and $r_{K}$  to respectively represent:
\[
R_{K}=\max\{|x|: x\in K\}, \quad r_{K}=\min\{|x|:x\in K\}.
\]

Now, we give the following result.
\begin{lem}\label{lemma upper bound}
	Let $K_i\in \ko$ and $\mu$ be a nonzero finite Borel measure not concentrated in any closed hemisphere. Let $R_{K_i}= h_{K_i}(v_i)$ for some $v_i\in \sn$. Assume $v_i\rightarrow v_0$ for $v_{0}\in \sn$. Then there exists $C_{\mu, v_0}\in (0,1)$ and $\widetilde{C}_{\mu, v_0}\in \mathbb{R}$, independent of $i$, such that
	\begin{equation}
		\frac{1}{|\mu|}\int_{\sn} \log h_{K_i}d\mu \geq \log r_{K_i} + C_{\mu, v_0}\log \frac{R_{K_i}}{r_{K_i}}+\widetilde{C}_{\mu, v_0}.
	\end{equation}
\end{lem}

\begin{proof}
	Consider $Q_i = K_i/r_{K_i}$. Then $h_{Q_i}\geq 1$ and $R_{Q_i}= R_{K_i}/r_{K_i} = h_{Q_i}(v_i)$. For $v_{0}\in \sn$ and $0<\alpha<1$, define
	\begin{equation*}
		\Omega_{\alpha} = \{v\in \sn: v\cdot v_0\geq \alpha\}.
	\end{equation*}
	Since $\mu$ is not concentrated in any closed hemisphere, there exists $\alpha_0>0$ (dependent on $\mu$ and $v_0$) such that $\mu(\Omega_{\alpha_0})>0$. Since $v_i\rightarrow v_0$, if $v\in \Omega_{\alpha_0}$, then for sufficiently large $i$, we have $v\cdot v_i\geq \alpha_0/2$. Then,  by the fact that $h_{Q_i}\geq 1$, we obtain
	\begin{equation*}
		\int_{\sn} \log h_{Q_i}d\mu \geq \int_{\Omega_{\alpha_0}}\log h_{Q_i}d\mu \geq  \int_{\Omega_{\alpha_0}}(\log R_{Q_i} + \log(\alpha_0/2))d\mu = (\log R_{Q_i} + \log(\alpha_0/2))\mu(\Omega_{\alpha_0}).
	\end{equation*}
	Hence,
	\begin{equation*}
		\frac{1}{|\mu|}\int_{\sn}\log h_{K_i}d\mu = \log r_{K_i} + \frac{1}{|\mu|} \int_{\sn} \log h_{Q_i}d\mu \geq \log r_{K_i}+\frac{\mu(\Omega_{\alpha_0})}{|\mu |} \log \frac{R_{K_i}}{r_{K_i}}+\log(\alpha_0/2) \frac{\mu(\Omega_{\alpha_0})}{|\mu |}.
	\end{equation*}
The proof is completed.
\end{proof}

\section{Solutions to the generalized Gaussian log-Minkowski problem}
\label{Sec7}
In this section,  we are devoted to proving  Theorem \ref{mtheo}. The key is to demonstrate that the minimizer of the optimization problem
\begin{equation*}\label{min}
\inf\{\Phi_{\mu}(K):\gamma_{b,m}(K)=\kappa_{0}\in (3/4, 1) \ {\rm and} \ K \in \ko\}
\end{equation*}
 exists, where $\Phi_{\mu}(\cdot)$ is defined by \eqref{phi}.  To do that, $C^{0}$ estimate is first required.

\begin{lem}\label{lemma lower bound}
	Let $K\in \ko$. If $\gamma_{b,m}(K)= \kappa_0\in (3/4, 1)$, then there exists $c_0>0$ (dependent of $\kappa_0)$ such that $h_{K}\geq c_0$.
\end{lem}

\begin{proof}
	We take a contradictory technique. If not, suppose that there exists a subsequence $K_{i}\in \ko$, $v_i\in \sn$ and $c_i\rightarrow 0^+$ as $i\rightarrow \infty$, such that
	\begin{equation*}
		K_{i}\subset \{x\in \rnnn: x\cdot v_i\leq c_i\}.
	\end{equation*}
	Hence,
	\begin{equation*}
		\gamma_{b,m}(K_{i}) \leq \gamma_{b,m}(\{x\in \rnnn: x\cdot v_i\leq c_i\})\rightarrow 1/2,
	\end{equation*}
	as $i\rightarrow \infty$. This contradicts to the given condition that $\gamma_{b,m}(K_{i})=\kappa_0\in (3/4,1)$. Hence, the proof is finished.
\end{proof}

Next, we will give the solvability of related optimization problem.

\begin{lem}\label{lemma existence of minimizer}
	If $\mu$ is a nonzero finite Borel measure on $\sn$ that is not concentrated in any closed hemisphere, then there exists $K_0\in \mathcal{K}_o^n$ that solves the optimization problem \eqref{eq 1}.
\end{lem}

\begin{proof}
	Taking a minimizing sequence $K_i\in \ko$. By Lemma \ref{lemma lower bound}, we know that there exists $c_0>0$ (independent of $i$) such that $h_{K_i}\geq c_0$.
	
	Next, we first prove that $r_{K_i}$ is uniformly bounded from above. If not, then there exists a subsequence $K_{i_j}$ and $r_{K_{i_j}}\rightarrow \infty$ as $j\rightarrow \infty$ such that $K_{i_j}\supset r_{K_{i_j}} B^{n}$. This implies $\gamma_{b,m}(K_{i_j})\geq \gamma_{b,m}(r_{K_{i_j}}B^{n})\rightarrow 1$ as $j\rightarrow \infty $, which is a contradiction to the fact $\gamma_{b,m}(K_{i_j}) = \kappa_{0}\in(3/4,1)$. So $r_{K_{i}}\leq c_{1}$ for $c_{1}>0$.  Let  $v_{i}$ be the maximum point of $h_{K_{i}}$ for some $v_{i}\in \sn$, in view of the fact that sphere is compact,  we can choose a subsequence of $K_{i}$ (still denoted by $K_{i}$)  such that $v_{i}\rightarrow v_{0}$ for $v_{0}\in \sn$. Thus combining with Lemma \ref{lemma upper bound}, we have
	\begin{equation}
\begin{split}
\label{upper}
		\Phi_{\mu}(K_{i})=\frac{1}{|\mu|}\int_{\sn} \log h_{K_i}d\mu &\geq \log r_{K_i} + C_{\mu, v_0}\log \frac{R_{K_i}}{r_{K_i}}+\widetilde{C}_{\mu, v_0}\\
&\geq \log c_0+C_{\mu, v_0}(\log R_{K_i}-\log c_{1})+\widetilde{C}_{\mu, v_0},
\end{split}
\end{equation}
 where $C_{\mu,v_{0}}$ and $\widetilde{C}_{\mu,v_{0}}$ come from Lemma \ref{lemma upper bound}.   On the other hand, we can choose a suitable $r_{0}>0$ such that $\gamma_{b,m}(r_{0}B^{n})=\kappa_{0}$ with
\[
\Phi_{\mu}(r_{0} B^{n})=\log r_{0},
\]
this reveals that $\Phi_{\mu}(K_{i})$ is bounded from above by a constant. Together with \eqref{upper}, this implies that $R_{K_i}$ is bounded from above. Then, by means of Blaschke's selection theorem,  $K_{i}$ has a convergent subsequence, denoted again by $K_{i}$, whose limit is called as $K_{0}$. That $h_{K_i}\geq c_0$ now implies that $K_0\in \ko$. By the continuity of all functionals involved, we conclude that $K_0$ is a minimizer to \eqref{eq 1}. The proof of this lemma is completed.
\end{proof}

{\bf Proof of Theorem \ref{mtheo}. }  This theorem holds by combining Lemma \ref{lemma existence of minimizer} and  Lemma \ref{lemma euler lagrange}.

\section*{Acknowledgment}The author would like to thank professors Yong Huang, Yiming Zhao and the referee for their valuable comments on this work.


\begin{thebibliography}{20}



\bibitem{BLYZ12}
K. B\"{o}r\"{o}czky, E. Lutwak, D. Yang\ and\ G. Zhang, The logarithmic Minkowski problem, J. Amer. Math. Soc. {\bf 26} (2013), no.~3, 831--852.


\bibitem{CHW23}
S. Chen, S. Hu, W. Liu\ and\ Y. Zhao, On the planar Gaussian-Minkowski problem, Adv. Math. {\bf 435} (2023), part A, Paper No. 109351.


\bibitem{FH23}
Y. Feng, S. Hu\ and\ L. Xu, On the $L_p$ Gaussian Minkowski problem, J. Differential Equations {\bf 363} (2023), 350--390.


\bibitem{F62}
W. J. Firey, $p$-means of convex bodies, Math. Scand. {\bf 10} (1962), 17--24.


\bibitem{G06}
R. J. Gardner, {\it Geometric tomography}, second edition, Encyclopedia of Mathematics and its Applications, 58, Cambridge University Press, New York, 2006.



    \bibitem{HLYZ10}
C. Haberl,  E. Lutwak, D. Yang\ and\ G. Zhang, The even Orlicz Minkowski problem, Adv. Math. {\bf 224} (2010), no.~6, 2485--2510.

     \bibitem{HXYZ21}
    Y. Huang, D.~M. Xi\ and\ Y. Zhao, The Minkowski problem in Gaussian probability space, Adv. Math. {\bf 385} (2021), Paper No. 107769, 36 pp.


\bibitem{HLYZ16}
Y. Huang, E. Lutwak, D. Yang\ and\ G. Zhang, Geometric measures in the dual Brunn-Minkowski theory and their associated Minkowski problems, Acta Math. {\bf 216} (2016), no.~2, 325--388.

 \bibitem{IE23}
    M.~N. Ivaki\ and\ E. Milman, Uniqueness of solutions to a class of isotropic curvature problems, Adv. Math. {\bf 435} (2023), part A, Paper No. 109350.


\bibitem{KL23}
L. Kryvonos\ and\ D. Langharst, Weighted Minkowski's existence theorem and projection bodies, Trans. Amer. Math. Soc. {\bf 376} (2023), no.~12, 8447--8493.

 \bibitem{Lu22}
J. Liu, The $L_p$-Gaussian Minkowski problem, Calc. Var. Partial Differential Equations {\bf 61} (2022), no.~1, Paper No. 28, 23 pp.


 \bibitem{JT24}
J. Liu\ and\ S. Tang, The generalized Gaussian Minkowski problem, J. Geom. Anal. {\bf 34}, 302 (2024).

\bibitem{ELYZ12}
E. Lutwak, S. Lv, D. Yang\ and\ G. Zhang, Extensions of Fisher information and Stam's inequality, IEEE Trans. Inform. Theory {\bf 58} (2012), no.~3, 1319--1327.


\bibitem{S14}
R. Schneider, {\it Convex bodies: the Brunn-Minkowski theory}, second expanded edition, Encyclopedia of Mathematics and its Applications, 151, Cambridge University Press, Cambridge, 2014.

\bibitem{Sh23}
W. Sheng\ and\ K. Xue, Flow by Gauss curvature to the $L_{p}$-Gaussian Minkowski problem, arXiv: 2212.01822.




\end{thebibliography}
\end{document}